\documentclass[10pt,preprint]{amsart}
\usepackage{amssymb,amsfonts,amsthm,amsmath}
\usepackage{hyperref}
\usepackage{tikz}

\def\N{\mathbb N}

\theoremstyle{plain}
\newtheorem{theorem}{Theorem}
\newtheorem{lemma}{Lemma}

\theoremstyle{definition}

\begin{document}

\title{Equable Triangles on the Eisenstein Lattice}

\author{Christian Aebi and Grant Cairns}

\address{Coll\`ege Calvin, Geneva, Switzerland 1211}
\email{christian.aebi@edu.ge.ch}
\address{Department of Mathematical and Physical Sciences, La Trobe University, Melbourne, Australia 3086}
\email{G.Cairns@latrobe.edu.au}


\maketitle

\section{Introduction} 
This paper concerns triangles whose vertices lie on the \emph{Eisenstein lattice}, which is the lattice in the complex plane generated by the elements $1$ and $\omega=-\frac12+i\frac{\sqrt3}2$.
We investigate triangles that are \emph{equable}, that is, they have equal perimeter and area. Our study is inspired by the classification of equable Heron triangles, which have integer sides and area and are known to be realisable on the integer lattice \cite{Yiu}. There are precisely five triangles of the preceding type, up to Euclidean motions; see  \cite{Fo} and \cite{Br}. By comparison, on the Eisenstein lattice, we find the following result.

\begin{theorem} There are only two equable triangles having vertices on the Eisenstein lattice, up to Euclidean motions. They are realized by the following vertices.
\begin{enumerate}
\item[\rm (a)] $A=8+4\omega,\ B= 4+8\omega,\ C=0$,

\item[\rm (b)] $A=6+3\omega,\ B= 8+16\omega,\ C=0$.
\end{enumerate}
\end{theorem}

Figure \ref{F} shows the two triangles, with the first (equilateral) triangle translated by $5$ to the right.

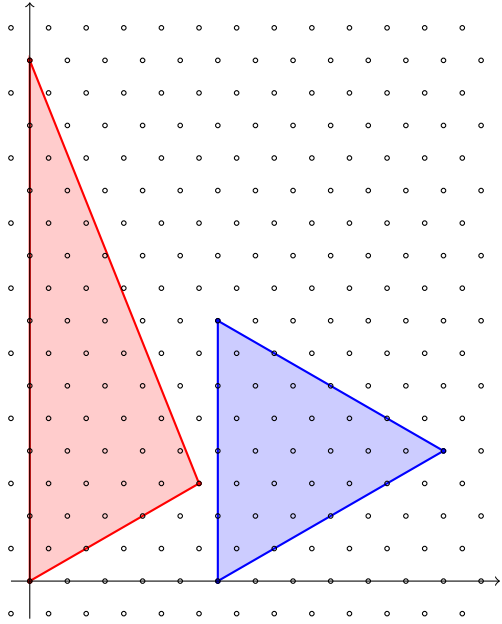
\begin{figure}
\begin{tikzpicture}[scale=.5]
\def\r{1.732};

\draw [blue,fill] (5,0) circle (.06);
\draw [blue,fill] (5+6, 2*\r) circle (.06);
\draw [blue,fill] (5,4*\r) circle (.06);

\draw [red,fill] (0,0) circle (.06);
\draw [red,fill] (9/2, 3*\r/2) circle (.06);
\draw [red,fill] (0,8*\r) circle (.06);

\draw [fill, red!20] (0,0) 
  -- (9/2, 3*\r/2)  
  -- (0,8*\r) 
  -- cycle;

\draw [thick, red] (0,0) 
  -- (9/2, 3*\r/2)  
  -- (0,8*\r) 
  -- cycle;

\draw [fill, blue!20] (5,0) 
  -- (5+6, 2*\r) 
  -- (5,4*\r)
  -- cycle;

\draw [thick, blue] (5,0) 
  -- (5+6, 2*\r) 
  -- (5,4*\r)
  -- cycle;

\foreach \i in {0,...,12}
\foreach \j in {0,...,8}
\draw (\i,\j*\r) circle (.06);
\foreach \i in {0,...,12}
\foreach \j in {-1,...,8}
\draw (\i-1/2,\r*\j+\r/2) circle (.06);

 \draw [->] (0,-1) -- (0,15.4);
 \draw [->] (-1/2,0) -- (12.5,0);

  \end{tikzpicture}
\caption{Two equable triangles  on the Eisenstein lattice}\label{F}
\end{figure}

\section{Proof of the Theorem} 

Consider an equable triangle $T$ with vertices $A=a_1+a_2\omega, B=b_1+b_2\omega$ and $C=0$, where $a_1,a_2,b_1,b_2\in\N\cup\{0\}$. Let $a,b,c$ denote the lengths of the sides $AC,BC,AB$, respectively.
Notice that the squares of the side lengths are integers; for example, $a^2=a_1^2-a_1a_2+a_2^2$. Moreover, the signed area of $T$ is $\frac{\sqrt3}4(a_1b_2-a_2b_1)$.

\begin{lemma}\label{L:bcint}
The side lengths $a,b,c$ are each of the form $\sqrt3n$, for some $n\in \N$.\end{lemma}

\begin{proof}
From above, as $T$ is equable, $a+b+c=\frac{\sqrt3}4(a_1b_2-a_2b_1)$, so
 \[
 \sqrt{3a^2}+\sqrt{3b^2}+\sqrt{3c^2}=\sqrt3(a+b+c)=\frac{3}4(a_1b_2-a_2b_1),
 \]
 which is rational. But it is well known that if $\sum_{i=1}^n\sqrt{m_i}$ is rational for integers $m_1,\dots,m_n$, then $\sqrt{m_i}$ is rational for each $i$; see for example \cite{art} or \cite{Yuan}. So $\sqrt{3a^2},\sqrt{3b^2},\sqrt{3c^2}$ are each rational. Hence, as $3a^2,3b^2,3c^2$ are integers, it follows that $\sqrt{3a^2},\sqrt{3b^2},\sqrt{3c^2}$ are also integers. So the side lengths $a,b,c$ are each of the form $\frac{n}{\sqrt3}$, for some $n\in \N$. Thus, since the squares of the side lengths are integers, the required result follows.
 \end{proof}

We follow the reasoning used in the proof of the  equable Heron triangle theorem given in the Appendix in \cite{AC1}.
For equable triangles with sides $a,b,c$, Heron's formula gives
\begin{equation}\label{E}
(a+b+c)(-a+b+c)(a-b+c)(a+b-c)=16(a+b+c)^2.
\end{equation}
Let $u=\frac{-a+b+c}{\sqrt3}, v=\frac{a-b+c}{\sqrt3},w=\frac{a+b-c}{\sqrt3}$, so that $u,v,w\in \N$ by the above lemma, and $a=\frac{\sqrt3(v+w)}{2},b=\frac{\sqrt3(u+w)}{2},c=\frac{\sqrt3(u+v)}{2}$.
Then Equation \eqref{E} gives 
\begin{equation}\label{E:1}
3uvw=16(u+v+w),
\end{equation}
and we may assume without loss of generality that $u\le v\le w$. 
 Note that by construction, $u,v,w$ have the same parity, and  from \eqref{E:1}, $u,v,w$ are necessarily even. Let $u=2x,v=2y,w=2z$, so $a=\sqrt3(y+z),b=\sqrt3(x+z),c=\sqrt3(x+y)$.
 Then
$3xyz=4(x+y+z)$.
Thus
\[
y\le z=\frac{4(x+y)}{3xy-4},
\]
so
$3xy^2-8y-4x\le 0$.
Hence 
\[
x\le y\le \frac{4+\sqrt{16+12x^2}}{3x}\]
so
$3x^2\le 4+\sqrt{16+12x^2}$.
Hence
$(3x^2-4)^2\le (16+12x^2)$.
Thus
$9x^4-36x^2\le 0$,
which gives $x\le 2$. Then 
\[
y\le \frac{4+\sqrt{16+12x^2}}{3x} \le \frac{4+\sqrt{16+12}}{3},
\]
 since the function $\frac{4+\sqrt{16+4x^2}}{3x} $ is decreasing for positive $x$. So, as $y$ is an integer,  $y\le 3$. Then, considering the values $x\le 2$, $y\le 3$ and 
$z=\frac{4(x+y)}{3xy-4}$, we find there are only two solutions, which have the following integer values for $x,y,z$:
\[
1,2,6\quad\text{and}\quad
2,2,2,
\]
which imply finally the values for $a,b,c$:
\[
8\sqrt3,7\sqrt3,3\sqrt3\quad\text{and}\quad
4\sqrt3,4\sqrt3,4\sqrt3.
\]




\bibliographystyle{amsplain}

\begin{thebibliography}{}

\bibitem{Yiu} Paul Yiu, Heronian triangles are lattice triangles,
\emph{Amer. Math. Monthly} {\bf 108} (2001),
no. 3, 261--263.

\bibitem{Fo}  Arthur H. Foss, Integer-sided triangles, \emph{Math. Teacher} {\bf  73} (1980), no. 5,  390--392.

\bibitem{Br}  Christopher J. Bradley, 
\emph{Challenges in geometry},
 {Oxford University Press, Oxford},
     2005.

\bibitem{art} Victor Wang,  \emph{The Art of problem Solving}, \url{https://artofproblemsolving.com/community/c1461h1035155}.

\bibitem{Yuan} Qiaochu Yuan,  \emph{Annoying Precision}, \url{https://qchu.wordpress.com/2009/07/02/square-roots-have-no-unexpected-linear-relationships/}.

\bibitem{AC1}
Christian Aebi and Grant Cairns,  \emph{Lattice equable quadrilaterals I: parallelograms}, L'Enseign.~Math.  \textbf{67} (2021), no.~3/4,  369--401.

\end{thebibliography}
{}

\end{document}